\documentclass{article}

\usepackage{amsmath, amssymb, amsfonts, amsthm, enumerate, float, lineno, mathrsfs, tikz}

\usepackage{authblk}

\usepackage{caption}
\usepackage{subcaption}

\usepackage{multirow}
\usepackage{hhline}

\usepackage{geometry}
\renewcommand{\S}{\mathcal{S}}

\newtheorem{proposition}{Proposition}
\newtheorem{lemma}{Lemma}
\newtheorem{theorem}{Theorem}
\newtheorem{corollary}{Corollary}
\theoremstyle{definition}

\title{Rooted forests that avoid sets of permutations}
\author{Katie Anders\thanks{kanders@uttyler.edu} }
\author{Kassie Archer\thanks{karcher@uttyler.edu}}
\affil{Department of Mathematics, The University of Texas at Tyler, 3900 University Blvd, Tyler, TX 75799}
\date{}

\begin{document}

\maketitle

\begin{abstract}
We say that an unordered rooted labeled forest avoids the pattern $\pi\in\S_n$ if the sequence obtained from the labels along the path from the root to any vertex does not contain a subsequence that is in the same relative order as $\pi$. We enumerate several classes of forests that avoid certain sets of permutations, including the set of unimodal forests, via bijections with set partitions with certain properties. We also define and investigate an analog of Wilf-equivalence for forests.
\end{abstract}


\section{Introduction}

 A permutation $\pi\in\S_n$ is said to avoid a pattern $\sigma\in\S_k$ if there is no subsequence in $\pi$ that is in the same relative order as $\sigma$. 
This idea of pattern avoidance has been generalized to several other classes of combinatorial objects, including words, partitions, matchings, graphs, posets, and trees. In fact, pattern avoidance has been defined and studied for many different classes of trees and forests in several different contexts. Usually, pattern avoidance is defined in terms of sets of trees that avoid other trees, as in the case of unlabeled binary trees \cite{Rowland10}, unlabeled ternary trees \cite{GPPT12}, and planar rooted trees with labeled leaves \cite{Dot12}. In computer science, the number of occurrences of certain subtree patterns are of interest, as in \cite{FS80, SF83}. 
Occasionally, the objects of study are trees or forests that avoid a permutation instead of another tree. One example can be found in \cite{LPRS16} in which planar increasing trees (which the authors refer to as \textit{heaps}) are said to avoid a permutation if the sequence obtained by reading the labels of the vertices during a breadth-first search (i.e. reading across) avoids that permutation. 

Our objects of study will be unordered (i.e. non-planar) rooted forests. Properties of these forests have previously been studied extensively. In particular, there are many interesting results regarding statistics on these trees and forests, such as descents \cite{Ges96, Gon16}, major index \cite{LiangWachs92}, inversions \cite{MalRio68, LiangWachs92}, leaves \cite{Ges96}, hook length \cite{CGG09, FerGou13}, leaders (analogs of right-to-left minima) \cite{GesSeo04, SeoShin07, Hou16}, and many others. Additionally, increasing trees and forests, which avoid the pattern $21$ in our context, have been widely studied and are useful combinatorial objects, as in \cite{KPP94,KMPW14,BFS92}, and are easily enumerated (see \cite{Stanley}). Alternating trees, which avoid the consecutive patterns $321$ and $123$ in our context, have also been enumerated in \cite{CDR01}. 

In Section~\ref{sec:prelim}, we introduce the necessary background and definitions, including what it means for an unordered rooted labeled forest to avoid a given permutation or set of permutations. In Section \ref{sec:Wilf}, a analog of Wilf-equivalence is defined for forests. Specifically, it is determined that there are two forest-Wilf-equivalence classes for patterns of length three, in contrast to the case of permutations where there is only one Wilf-equivalence class. In the remaining sections, 
we provide enumerations for the sets of forests on $[n]$ that avoid given sets of permutations of lengths 3 and 4 by constructing bijections with certain types of partitions of $[n]$ into sets or lists with special properties. The table in Figure~\ref{fig:results} summarizes the results obtained in this paper. 
\begin{figure}
\centering
\renewcommand{\arraystretch}{1.25}
\begin{tabular}{|l|l|l|}
\hline
Result & Forests & Enumeration\\ \hline
Theorem~\ref{thm:unimodal} & $F_n(213,312)$ & \multirow{2}{*}{ $\displaystyle\sum_{k=1}^n k! \,c(n,k)$ }\\ \cline{1-2}
 Corollary~\ref{cor:V} & $F_n(231, 132)$& \\ \hline
 Theorem~\ref{thm:uni - 123} & $F_n(213,312, 123)$ & \multirow{2}{*}{ $\displaystyle\sum_{k=1}^n \mathcal{B}(n)\,c(n,k)$ }\\ \cline{1-2}
 Corollary~\ref{cor:unimodal plus} & $F_n(231, 132, 321)$& \\ \hline
 Theorem~\ref{thm:uni -321} & $F_n(213,312, 321)$ & \multirow{2}{*}{ $\displaystyle\sum_{k=1}^n k! \,S(n,k)$ }\\ \cline{1-2}
 Corollary~\ref{cor:unimodal plus} & $F_n(231, 132, 123)$& \\ \hline
 Theorem~\ref{thm:uni plus} & $F_n(312, 213, 132)$&  \multirow{4}{*}{ $\displaystyle  \sum_{k=1}^n \frac{n!}{k!}{{n-1}\choose{k-1}}$} \\ \cline{1-2}
 Corollary~\ref{cor:Vplus} & $F_n(132, 231, 312)$&  \\ \cline{1-2}
 Theorem~\ref{thm:one descent plus}  & $F_n(321, 132, 213)$&\\ \cline{1-2}
  Theorem~\ref{thm:one descent plus} & $F_n(123, 312, 231)$&\\ \hline
     Theorem~\ref{thm:uni+231} & $F_n(213, 312, 231)$ &  \multirow{2}{*}{ recurrence (see theorem statement) } \\ \cline{1-2}
  Corollary~\ref{cor:V+213} & $F_n(231, 132, 213)$&\\ \hline
   Theorem~\ref{thm:one descent} & $F_n(321, 2143, 3142)$ &  \multirow{2}{*}{ $\displaystyle n!  + \sum \frac{n!}{2^{\ell}} {{n-k-1}\choose{\ell-1}}{{k}\choose{\ell}}$} \\ \cline{1-2}
  Corollary~\ref{cor:downdown} & $F_n(123, 3412, 2413)$&\\ \hline
\end{tabular}
\caption{For each result listed, the sets of forests that avoid the certain sets of permutations are enumerated. The sum associated to Theorem~\ref{thm:one descent} and Corollary~\ref{cor:downdown} occurs over all $1\leq \ell\leq k\leq n$ so that $\ell+k\leq n$.}
\label{fig:results}
\end{figure}

\section{Preliminaries}\label{sec:prelim}

\subsection{Permutations, partitions, and compositions}

Let $\S_n$ denote the set of permutations on the set $[n] = \{1,2,\ldots, n\}$. We write a permutation in its one-line notation as $\pi_1\pi_2\ldots \pi_n$, and we write a permutation in its cycle notation as a product of disjoint cycles. For example,  $\pi = 4672513$ is a permutation written in its one-line notation, and $\pi = (1426)(37)(5)$ is the same permutation written in its cycle notation. The unsigned Stirling numbers of the first kind, denoted by $c(n,k)$ (or ${ n\brack k}$), enumerate the permutations in $\S_n$ which can be decomposed into $k$ disjoint cycles. 

A \textit{descent} of a permutation of $[n]$ (or more generally any ordered set) is an element $i$ so that $\pi_i>\pi_{i+1}$. An \textit{ascent} is an element $i$ so that $\pi_i<\pi_{i+1}$. An element $\pi_i$ is called a \textit{left-to-right maximum} if $\pi_i>\pi_j$ for all $j<i$, and we say $\pi_i$ is a \textit{left-to-right minimum} if $\pi_i<\pi_j$ for all $j<i$. The \textit{inverse} of a permutation $\pi \in \S_n$ is the permutation $\pi^{-1}$ defined by $\pi^{-1}_i = j$ if and only if $\pi_j = i$. The \textit{reverse} of a permutation $\pi\in\S_n$ is the permutation $\pi^r$ that satisfies $\pi^r_i = \pi_{n+1-i}$ for all $i\in[n]$. For a permutation $\pi\in\S_n$, we define the \textit{complement} $\pi^c$ to be the permutation that satisfies $\pi_i^c = n+1-\pi_i$. It is easy to check that for any permutation $\pi\in\S_n$, these trivial symmetries satisfy the relationship $(\pi^{-1})^r = (\pi^c)^{-1}$.

A permutation $\pi  \in \S_n$ is said to \textit{contain} the pattern $\sigma \in\S_k$ if there is some sequence $i_1<i_2<\cdots<i_k$ so that $\pi_{i_1}\pi_{i_2}\ldots\pi_{i_k}$ is in the same relative order as $\sigma_1\sigma_2\ldots \sigma_k$.
For example, the permutation $\pi = 51263748$ contains the pattern $\sigma = 231$ since there is a subsequence (for example, $\pi_4\pi_6\pi_7 = 674$) in the same relative order as $\sigma$. We say a permutation $\pi$ \textit{avoids} the pattern $\sigma$ if there is no such subsequence. As an example, the permutation $\pi$ above avoids $\sigma = 321$. 

Many of the things that we have defined for permutations on $[n]$ can be applied to permutations on any finite linearly ordered set $A$. For example, $\pi = 59381$ is a permutation on the set $A = \{1,3,5, 8, 9\}$. It has descents at positions 2 and 4. The left-to-right maxima are 5 and 9, and the left-to-right minima are 5, 3, and 1. The reverse of $\pi$ is $\pi^r = 18395$, the ``complement'' is $\pi^c = 51839$ (obtained by replacing the smallest with the largest, the second smallest with the second largest, etc.), and the ``inverse'' of $\pi$ is $\pi^{-1} = 95183$. The easiest way to obtain this is by taking the permutation in the same relative order as $\pi$, which in this case is 35241, taking the inverse, and replacing the appropriate values with values from $A$. Also, notice that $\pi$ contains 231 and avoids 123. These ideas will be useful at times when we are only concerned with permutations of a subset of $[n]$.

A \textit{set partition} of $[n]$ is a collection of disjoint subsets of $[n]$ whose union is $[n]$. An \textit{ordered} set partition is a set partition in which the order of the sets matters. A partition of $[n]$ into \textit{lists} is a set partition in which the elements of the subsets are ordered. 
For example, $\{1,3,4,5\} \{2,6\}$ and $\{1,6\} \{2\} \{3,4,5\}$ are two different set partitions of $[6]$. As ordered set partitions, $\{1,6\} \{2\} \{3,4,5\}$ and $ \{2\} \{1,6\} \{3,4,5\}$ would be considered different, and as partitions of $[6]$ into lists, $\{1,6\} \{2\} \{3,4,5\}$ and $\{6,1\} \{2\} \{3,5,4\}$ would be considered different.

A \textit{composition} of $n$ is a list $\lambda = (\lambda_1, \lambda_2, \ldots, \lambda_k)$ so that $\lambda_1+ \lambda_2 + \cdots + \lambda_k = n$. For example, $(3,3,4)$, $(3,4,3)$, and $(2,2,1,4,1)$ are three different compositions of $10$. It is well known that there are exactly ${{n-1}\choose{k-1}}$ compositions of $n$ into $k$ parts (see \cite{Stanley}). 

\subsection{Rooted forests}
Let $F_n$ denote the set of unordered rooted labeled forests on $[n]$. We draw forests as rooted trees with an unlabeled root as in Figure~\ref{fig:rooted tree}. 

\begin{figure}[h]
\centering
\includegraphics{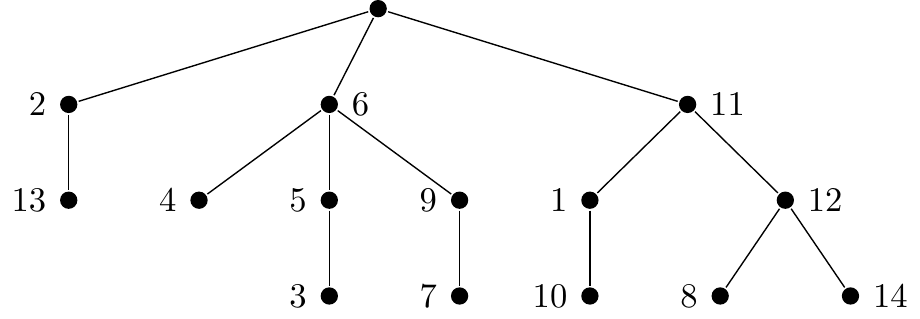}
  \caption{An example of a rooted forest on $[14]$. There are three trees in this forest with roots labeled 2, 6, and 11. }
  \label{fig:rooted tree}
  \end{figure}

A vertex $j$ is a \textit{descendant} of vertex $i$ if $i$ appears in the unique path from the root to $j$. In this case, we say that $i$ is an \textit{ancestor} of $j$. If $i$ and $j$ are adjacent in the forest, then we say that $i$ is the \textit{parent} of $j$ and that $j$ is a \textit{child} of $i$. For example, in Figure~\ref{fig:rooted tree}, vertices 11 and 12 are ancestors of 14, while vertex 7 is a child of 9 and a descendant of both 6 and 9. 

A vertex $i$ is a \textit{top-down maximum} for a forest if it is larger than all of its ancestors. For example, in Figure~\ref{fig:rooted tree}, seven of the vertices are top-down maxima, namely 2, 6, 9, 11, 12, 13, and 14. Notice that roots of trees in a forest are always top-down maxima since they have no ancestors. An \textit{increasing forest} is one for which all vertices are top-down maxima, as in Figure~\ref{fig:inc varphi}. 

We say that $F'$ is a \textit{subforest} of $F$ if $F'$ contains  a subset of vertices from $F$ so that if vertex $i$ appears in $F'$, then so does every ancestor of $i$. The largest increasing subforest of the forest in Figure~\ref{fig:rooted tree} is the one containing vertices 2, 6, 9, 11, 12, 13, and 14. 

We say that a labeled rooted forest on $[n]$ avoids the pattern $\sigma\in\S_k$ if along each path from the root to a vertex, the sequence of labels $i_1i_2i_3\ldots i_m$ do not contain a subsequence that is in the same relative order as $\sigma_1\sigma_2\ldots \sigma_k$. 
Let $F_n(\sigma)$ denote the set of forests on $[n]$ that avoid $\sigma$ and let $f_n(\sigma)$ be the number of such forests. Notice that $F_n(21)$ is the set of increasing forests on $[n]$ and $F_n(12)$ is the set of decreasing forests on $[n]$.

  It is well known that the number of increasing forests on $[n]$ is $f_n(21) = n!$. (Similarly, the number of decreasing forests is $f_n(12) = n!$.) Perhaps the easiest way to see this is via induction. Any increasing forest on $[n]$ must have the label $n$ assigned to a leaf. On a forest of $n-1$ elements, there are $n$ places one can add such a leaf (as a child of the $n-1$ vertices or as a root). Thus $f_n(21) = n\cdot f_{n-1}(21)$. Since there is one forest of size one, there must be $n!$ increasing forests. A natural bijection, which we call $\varphi$, between permutations on an ordered set $A$ and increasing forests on $A$ can be described in the following way (as in \cite{Stanley}): Given $\pi\in \S_n$, the forest $\varphi(\pi)$ is obtained by letting all left-right minima be roots of the trees. For the remaining vertices, let $i$ be the child of the rightmost element $j$ of $\pi$ that precedes $i$ and is less than $i$.

   In Figure~\ref{fig:inc varphi}, we see an example of this bijection. If $\pi = 3,6,8,4,1,10,2,9,7,5$, then the left-to-right minima 3 and 1 will be roots of trees in the increasing forest. Since 3 is the rightmost element to the left of 6 that is smaller than it, 6 will be a child of 3. Similarly, 8 will be a child of 6. Since 3 is the rightmost element to the left of 4 that is smaller than it, 4 will be a child of 3, and so on. The inverse of this bijection is to traverse the forest clockwise starting at the root (after ordering the children of each vertex least to greatest from left to right as in Figure~\ref{fig:inc varphi}), reading off each label as it is reached. 
  
\begin{figure}[h]
\centering
\includegraphics{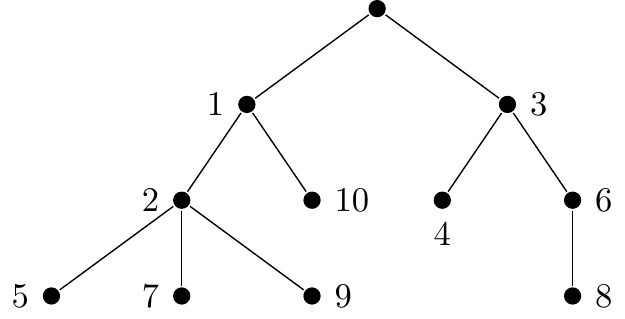}
  \caption{The map $\varphi$ sends a permutation of $[n]$ to an increasing forest on $[n]$. The forest pictured here is $\varphi(3,6,8,4,1,10,2,9,7,5)$.}
  \label{fig:inc varphi}
  \end{figure}
  


\section{Forest-Wilf-equivalence for patterns of size 3}\label{sec:Wilf}

An analog of Wilf-equivalence can be defined for rooted forests as follows. The pattern $\sigma$ is \textit{forest-Wilf-equivalent} to the pattern $\tau$ if $f_n(\sigma) = f_n(\tau)$. A set of patterns $S = \{\sigma_1, \ldots, \sigma_k\}$ is \textit{forest-Wilf-equivalent} to a set of patterns $T = \{\tau_1, \ldots, \tau_\ell\}$ if $f_n(\sigma_1, \ldots, \sigma_k) = f_n(\tau_1, \ldots, \tau_\ell)$. The trivial symmetry of complementation of permutations can be adapted to this setting. There is no clear analog of reverse or inverse for forests that preserves forest-Wilf-equivalence. 

 For a forest $F$ on $[n]$, let $F^c$ be the forest which has label $n+1-i$ at a vertex if and only if $F$ has label $i$ at the same vertex. Notice that we can define a map $\varphi_D$, which takes in a permutation $\pi\in\S_n$ and returns a decreasing forest on $[n]$, by applying $\varphi$ to get an increasing forest and then taking the complement of the forest obtained. 
\begin{proposition}\label{prop:comp}
For $n\geq 1$, $f_n(\sigma_1, \ldots, \sigma_m) = f_n(\sigma_1^c, \ldots, \sigma_m^c)$.
\end{proposition}
\begin{proof}
Clearly, $F \in F_n(\sigma_1, \ldots, \sigma_m)$ if and only if $F^c \in F_n(\sigma_1^c, \ldots, \sigma_m^c)$, and thus the result follows. 
\end{proof}

In this section, we determine that there are two forest-Wilf-equivalence classes for patterns of length three, in contrast to the case of permutations where there is only one Wilf-equivalence class. This is done by explicitly constructing a bijection between forests avoiding 312 and those avoiding 321. 
Theorems \ref{thm:uni plus} and \ref{thm:one descent plus} of this paper provide another example of a nontrivial forest-Wilf-equivalence of sets of patterns, namely that the sets $\{312, 213, 132\}$ and $\{321, 132, 213\}$ are forest-Wilf-equivalent.

By direct computation, it has been determined that $f_n(321)=f_n(231)$ for $n=1,2,3,4$.  However, $f_5(321)=918$, while $f_5(231)=917$.  The difference is that of the $60$ rooted forests with shape given in Figure \ref{fig: uneven tree on 5}, $43$ avoid $321$ while only $42$ avoid $231$.

\begin{figure}[h]
\centering
\includegraphics{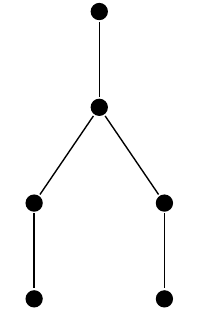}
  \caption{The shape of the rooted forest on $[5]$ that has a different number of forests avoiding $321$ than $231$}
  \label{fig: uneven tree on 5}
  \end{figure}

\begin{theorem}\label{thm:Wilf}
The patterns $321$ and $312$ are forest-Wilf-equivalent.
\end{theorem}

\begin{proof}
We generalize the proof of Simion and Schmidt from \cite{SimionSchmidt}. For a 312-avoiding forest, define a map $\alpha$ by the following process.  Fix all of the top-down maxima and, doing a breadth-first search, replace each non-top-down maximum element $i$ with the smallest element that is a descendant of $i$ and is not a top-down maximum, if there is such a descendant. Shift all other non-top-down maxima that are descendants of $i$ so that the previous relative order is maintained. Continue until each non-top-down maxima has been checked. Since all occurrences of a 321 pattern have been changed to an occurrence of a 312 pattern, the resulting forest is 321-avoiding.

To invert this process, we define a map $\beta$ as follows. Start with a 321-avoiding forest, fix all of the top-down maxima, and, doing a breadth-first search, replace each non-top-down maximum element $i$ with the largest element that is a descendant of $i$ and does not introduce a new top-down maximum, if such an element exists. As before, shift all other non-top-down maxima that are descendants of $i$ so that the previous relative order is maintained. Continue until each non-top-down maxima has been checked. Since all occurrences of a 312 pattern have been changed to an occurrence of a 321 pattern, the resulting forest is 312-avoiding.

These processes are inverses of each other. At the $k$th step of $\alpha$ or $\beta$, the vertex under consideration $v$ is either a top-down maxima or it is not. If it is, it is fixed under both maps. If it is not, only the elements of the subtree rooted at $v$ are permuted. 

Suppose $F$ is a 321-avoiding forest and $v$ is a not a top-down maximum. Let $A_1, A_2, \ldots, A_k$ be the sets of elements of the trees rooted at the $k$ children of $v$.  Suppose first that the vertex $v$ is the first non-top-down maximum encountered in a breadth-first search of the forest $F$. When $\beta$ is applied to $F$, the elements of the subtree of $F$ rooted at $v$ remain in that subtree. Under $\beta$, $v$ is replaced by its largest descendant $d_v$ that is smaller than $p_v$, the parent of $v$. If $a$ is a descendant of $v$ that is larger than $p_v$, and $a \in A_i$, then after this step, $a$ remains in the same place and in particular remains in $A_i$. If $b$ is a descendant of $v$ that is smaller than $p_v$ and $b\in A_\ell$, then $b$ replaces the next largest descendant of $v$ and may move to a new tree. If the next largest descendant of $v$ is in $A_j$, then $b\in A_j$ after this step. The element $v$ replaces the smallest descendant of $v$. The map $\beta$ will continue to permute the remaining vertices within their respective subtrees. In particular, $a$ remains in $A_i$ and $b$ remains in $A_j$.  

When $\alpha$ is applied to $\beta(F)$, the element $d_v$ is replaced with the smallest of its descendants, which must be $v$. Thus $\alpha\circ\beta$ fixes $v$. Notice that if the element $a$ is as above, $a$ remains fixed under this step of $\alpha$ and so $a \in A_i$ as before. Since $b<p_v$, $b$ will replace the next smallest descendant of $v$, and so $b$ will move back to $A_\ell$. Thus, each $A_i$ is fixed under $\alpha \circ \beta$. For each child of $v$, we can do the same process and each child of $v$ will be fixed under $\alpha\circ \beta$. We can continue in this manner for every non-top-down maximum of the forest.  

\end{proof}

It should be noted that this bijection preserves the top-down maxima and the shape of the forests.

\begin{figure}[h]
\centering
\includegraphics{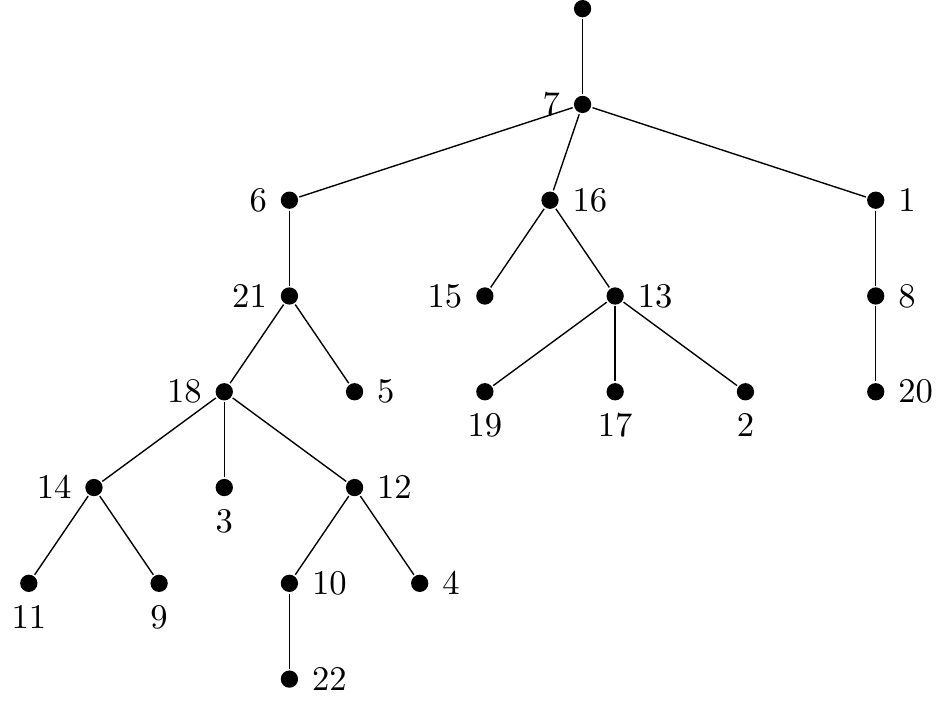}
  \caption{A rooted forest $F$ on $[22]$ that avoids 312. See Figure \ref{fig: AlphaOutput312avoiding} for $\alpha(F)$ where $\alpha$ is the map defined in the proof of Theorem \ref{thm:Wilf}.}
  \label{fig: AlphaInput312avoiding}
  \end{figure}

\begin{figure}[h]
\centering
\includegraphics{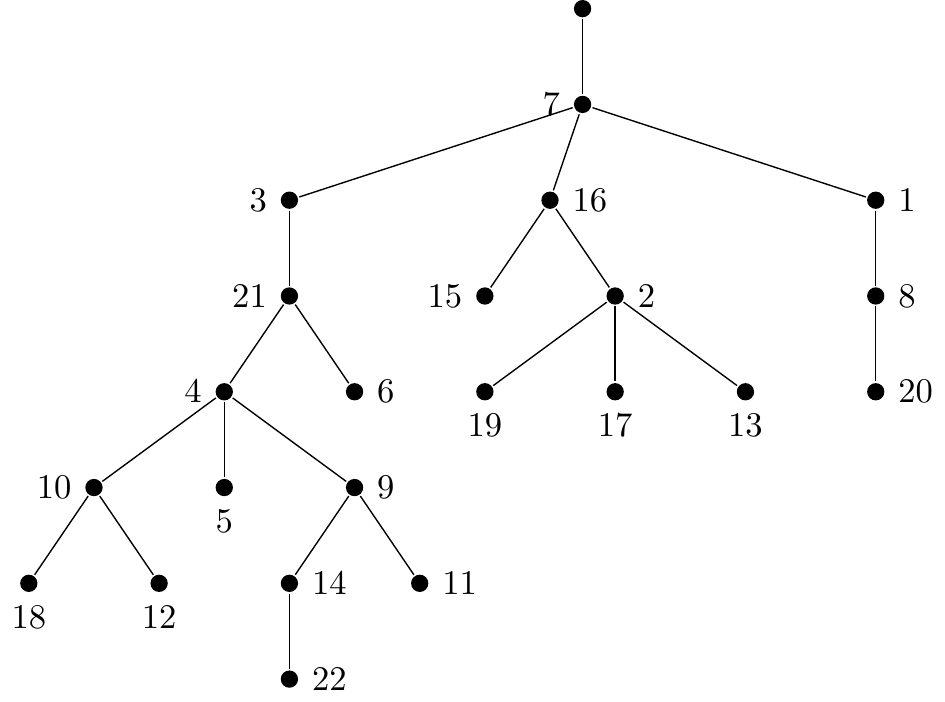}
  \caption{The 321-avoiding forest $\alpha(F)$ where $F$ is the 312-avoiding forest in Figure \ref{fig: AlphaInput312avoiding}.}
  \label{fig: AlphaOutput312avoiding}
  \end{figure}

  We can also define an analog of c-Wilf-equivalence for consecutive pattern avoidance to these forests. We say that two consecutive patterns are c-forest-Wilf-equivalent if the forests that avoid each of the consecutive patterns are equinumerous.
There are $3$ c-forest-Wilf-equivalence classes for consecutive pattern avoidance with patterns of length $3$, as can be seen in Figure \ref{fig:non-planar forests}.

\begin{figure}[h]
\centering
\renewcommand{\arraystretch}{1.25}
\begin{tabular}{c|c|c|c|c|c|c}
 &\multicolumn{3}{c|}{Classical} &\multicolumn{3}{c}{Consecutive}\\
 \hline
n & $321$ & $231$ & $132$ &$321$ & $231$ & $132$  \\
\hline
$1$ &$1$ &$1$ &$1$ &$1$ &$1$ &$1$\\
\hline
$2$ &$3$ &$3$ &$3$ &$3$ &$3$ &$3$\\
\hline
$3$ &$15$ &$15$ &$15$ &$15$ &$15$ &$15$\\
\hline
$4$ &$104$ &$104$ &$104$ &$107$ &$106$ &$106$\\
\hline
$5$ &$918$ &$917$ &$918$ &$997$ &$973$ &$972$\\
\end{tabular}
\caption{The number of forests on $\left[n\right]$ that avoid the given (classical or consecutive) pattern.}
\label{fig:non-planar forests}
\end{figure}



\section{Forests avoiding the set $\{312, 213\}$}\label{unimodal}

We say that a labeled rooted forest on $[n]$ is \textit{unimodal} if along each path from the root to a vertex, the labels are increasing, then decreasing. These are exactly the forests that avoid the permutations $312$ and $213$. 
For an example, see Figure~\ref{fig:unimodal}.

Recall that the unsigned Stirling numbers of the first kind, $c(n,k)$, count the number of permutations in $\S_n$ which can be decomposed into $k$ disjoint cycles. The set of ordered cycle decompositions of permutations in $\S_n$ are those cycle decompositions for which the order in which the cycles appear matters. 

\begin{theorem}\label{thm:unimodal}
For $n\geq 1$, $$f_n(312, 213) = \sum_{k=1}^n k! \,c(n,k).$$
\end{theorem}

\begin{proof}
We will construct a bijection, $\theta$, from the set of ordered cycle decompositions of permutations in $\S_n$ to unimodal forests. For some ordered cycle decomposition of a permutation $\pi\in\S_n$, write each cycle so that the maximum in each cycle is first. These maxima will correspond exactly to the top-down maxima of the corresponding forest. The top-down maxima in a unimodal forest constitute the largest increasing subforest of the forest. The order in which these cycles appear determines this increasing subforest by applying $\varphi$ to the ordered list of maxima. It remains to attach the remaining vertices. The elements in each cycle which are not the maximum are the descendants of that maximum in the forest. Use $\varphi_D$ on the remaining elements of each cycle to obtain an decreasing forest on those elements. Attach these to the maximum from that cycle. (See example in Figure~\ref{fig:unimodal}.)


This map is clearly invertible and thus each unimodal forest on $[n]$ with $k$ top-down maxima can be uniquely identified with a permutation on $[n]$ which can be decomposed into $k$ cycles, where the $k$ cycles appear in some order. Since a unimodal forest can have anywhere from 1 to $n$ top-down maxima and there are $k!$ ways to arrange the $k$ cycles, the result follows.
\end{proof}

\begin{figure}[h]
\centering
\includegraphics{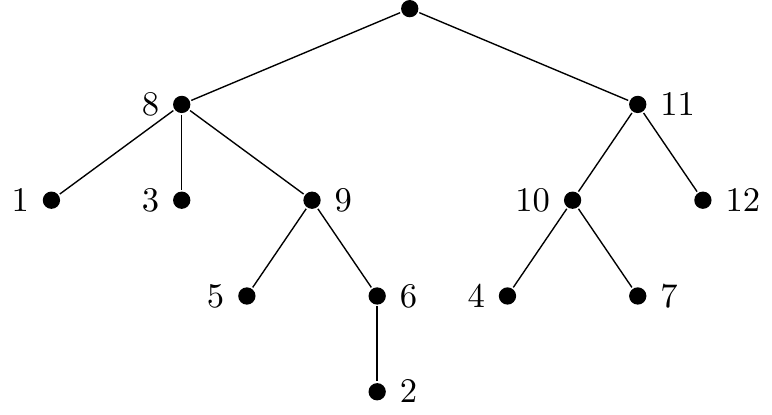}
  \caption{This example of a unimodal forest is realized as $\theta((11, 4, 10, 7)(12)(8,3,1)(9,5,2,6))$.}
  \label{fig:unimodal}
  \end{figure}

The forest in Figure~\ref{fig:unimodal} is obtained via the $\theta$ bijection from the following ordered cycle decomposition: $$(11, 4, 10, 7)(12)(8,3,1)(9,5,2,6).$$ Indeed, $8$, $9$, $11$, and $12$ are the maxima from each cycle and constitute the largest increasing subforest (of size 4) which one can obtain by applying $\varphi$ to the permutation $11, 12, 8, 9$. The remaining elements constitute decreasing trees (obtained via $\varphi$ and complementation) attached to each of these maxima. For example, we can obtain an increasing forest on $\{4, 10, 7\}$ by applying $\varphi$ to the permutation $4, 10, 7$. Take the complement of that forest (by sending 4 to 10 and 10 to 4) to obtain a decreasing forest. Finally, attach all trees in that decreasing forest to the node labeled 11.

\begin{corollary}
Let $1\leq k\leq n$. 
\begin{itemize}
\item The number of unimodal forests with exactly $k$ top-down maxima is equal to $k!c(n,k)$. 
\item The number of unimodal forests with exactly $m$ trees is equal to $$\sum_{k=m}^n  c(k,m) \,c(n,k).$$
\end{itemize}
\end{corollary}

\begin{proof}
The first statement follows directly from the proof of Theorem~\ref{thm:unimodal}. 
The second statement follows from the fact that there are exactly $c(k,m)$ increasing forests on $[k]$ with $m$ trees (since the $m$ roots correspond exactly to the left-to-right minima of the corresponding permutation and it is well-known that the unsigned Stirling numbers of the first kind count permutations by number of left-to-right minima). 
\end{proof}

By Proposition~\ref{prop:comp} and Theorem~\ref{thm:unimodal}, we obtain the following corollary pertaining to the forests that are \textit{reverse-unimodal}, i.e. if along each path from the root to a vertex, the labels are decreasing, then increasing. 
\begin{corollary}\label{cor:V}
For $n\geq 1$, $$f_n(132, 231) = \sum_{k=1}^n k! \,c(n,k).$$
\end{corollary}

\section{Unimodal forests avoiding one extra pattern}

In this section, we consider the forests avoiding the set $\{213, 312\}$ and one additional pattern of length 3. There are four such sets:  $\{213, 312,123\}$, $\{213, 312, 321\}$, $\{213, 312,132\}$, and $\{213, 312, 231\}$. We also find an enumeration for the set of permutations that avoid the patterns in the set $\{321, 132, 213\}$ by a very similar argument to those that avoid $\{213, 312,132\}$. 

Notice that the forests avoiding the set of patterns $\{213, 312, 123\}$ are exactly unimodal forests for which the increasing part of any unimodal path from the root to a vertex is at most length 2. Similarly, forests avoiding the set of patterns $\{213, 312, 321\}$ are exactly unimodal forests for which the decreasing part of any unimodal path from the root to a vertex is at most length 2. Unlike the corresponding permutations which avoid these sets, these sets of forests are not equinumerous. 

We start with a lemma regarding the height of an increasing forest. We say a forest has \textit{height} $h$ if the longest labeled path from the root of the forest to any vertex is of length $h$. For example, the forest in Figure~\ref{fig:unimodal} has height 4 since the longest path,  $8-9-6-2$, is of length 4. Recall that $\mathcal{B}(n)$ denotes the $n^{\text{th}}$ Bell number, counting the number of set partitions of $[n]$, and $S(n,k)$ denotes the Stirling number of the second kind, counting the number of ways to partition $[n]$ into $k$ non-empty subsets. 

\begin{lemma}\label{lem:height2}
For $n\geq1$, there are $\mathcal{B}(n)$ increasing forests on $[n]$ of height at most~2. There are $S(n,k)$ such forests with exactly $k$ trees.
\end{lemma}
\begin{proof}
We will construct a bijection, $\beta$, from the set of set partitions of $[n]$ to the set of increasing forests on $[n]$ of height at most 2. Let $P$ be a set partition of $[n]$. For each set $p_i\in P$, let the minimum element $j$ be a root of a tree in the forest and let all other elements of $p_i$ be children of $j$. Each part of the partition $P$ corresponds to a tree in the forest. This is clearly invertible.
\end{proof}

We will use this lemma to enumerate both sets of forests. 

\begin{theorem}\label{thm:uni - 123}
For $n\geq 1$, $$f_n(312, 213, 123) =  \sum_{k=1}^n \mathcal{B}(k)c(n,k).$$
\end{theorem}
\begin{proof}
We will construct a bijection, $\xi$, from the set of partitioned cycle decompositions of permutations in $\S_n$ to unimodal forests that avoid 123. For some partitioned cycle decomposition of a permutation $\pi\in\S_n$, write each cycle so that the maximum in each cycle is first. These maxima will correspond exactly to the top-down maxima of the corresponding forest. 
The partition of these cycles determines this increasing subforest by applying $\beta$ (defined in the proof of Lemma~\ref{lem:height2}) to the partition of the maxima. This gives us an increasing forests on the set of maxima of height at most 2. It remains to attach the remaining vertices. The elements in each cycle which are not the maximum are the descendants of that maximum in the forest. Use $\varphi_D$ on the remaining elements of each cycle to obtain an decreasing forest on those elements. Attach these to the maximum from that cycle. Since the largest increasing subforest has height at most two, the unimodal forest obtained avoids 123.  


This map is clearly invertible and thus each unimodal forest on $[n]$ which avoids 123 with $k$ top-down maxima can be uniquely identified with a permutation on $[n]$ which can be decomposed into $k$ cycles, where the $k$ cycles are partitioned into subsets. Since a unimodal forest can have anywhere from 1 to $n$ top-down maxima and there are $\mathcal{B}(k)$ ways to partition the $k$ cycles, the result follows.
\end{proof}

\begin{theorem}\label{thm:uni -321}
For $n\geq 1$, $$f_n(312, 213, 321) = \sum_{k=1}^n k!S(n,k).$$
\end{theorem}

\begin{proof}
We will construct a bijection, $\gamma$, from the set of ordered partitions of $[n]$ to unimodal forests that avoid 321. For some ordered partition of $[n]$ into $k$ parts, let $j_i$ be the maximum element of the $i$th part. These maxima will correspond exactly to the top-down maxima of the corresponding forest. The top-down maxima in a unimodal forest constitute the largest increasing subforest of the forest. The order in which these parts appear in the partition determines this increasing subforest by applying $\varphi$ to the ordered list of maxima. It remains to attach the remaining vertices. The elements in each part which are not the maximum are the children of that maximum in the forest.


This map is clearly invertible and thus each unimodal forest on $[n]$ which avoids 321 with $k$ top-down maxima can be uniquely identified with a partition on $[n]$ into $k$ parts, where the $k$ parts appear in some order. Since a unimodal forest can have anywhere from 1 to $n$ top-down maxima and there are $k!$ ways to arrange the $k$ parts, the result follows.
\end{proof}

By Proposition~\ref{prop:comp}, Theorem \ref{thm:uni - 123}, and Theorem~\ref{thm:uni -321}, we obtain the following corollary.
\begin{corollary}\label{cor:unimodal plus}
For $n\geq 1$, $$f_n(132, 231, 321) = \sum_{k=1}^n \mathcal{B}(k)c(n,k) \, \, \, \, \, \text{ and } \, \, \, \, \, f_n(132, 231, 123) = \sum_{k=1}^n k!S(n,k).$$
\end{corollary}

We now consider the forests avoiding all permutations in the set $\{312, 213, 132\}$. 
Notice that permutations that avoid these sets are unimodal where all elements of the increasing part are greater than those of the decreasing part. They are exactly the permutations that can be drawn on Figure~\ref{fig:312, 213, 132}.

\begin{figure}
\centering
\begin{subfigure}[b]{0.24\textwidth}
\centering
\begin{tikzpicture}
\draw (0,0) --(2,0) [thick, color=black!20] {};
\draw (0,0) --(0,2) [thick, color=black!20] {};
\draw (0,2) --(2,2) [thick, color=black!20] {};
\draw (2,0) --(2,2) [thick, color=black!20] {};
\draw (1,0) --(1,2) [thick, color=black!20] {};
\draw (0,1) --(2,1) [thick, color=black!20] {};
\draw (0,1) --(1,2) [thick] {};
\draw (1,1) --(2,0) [thick] {};
\end{tikzpicture}
\subcaption{$\{312, 213, 132\}$}
\label{fig:312, 213, 132}
\end{subfigure}
\begin{subfigure}[b]{0.24\textwidth}
\centering
\begin{tikzpicture}
\draw (0,0) --(2,0) [thick, color=black!20] {};
\draw (0,0) --(0,2) [thick, color=black!20] {};
\draw (0,2) --(2,2) [thick, color=black!20] {};
\draw (2,0) --(2,2) [thick, color=black!20] {};
\draw (1,0) --(1,2) [thick, color=black!20] {};
\draw (0,1) --(2,1) [thick, color=black!20] {};
\draw (0,1) --(1,0) [thick] {};
\draw (1,1) --(2,2) [thick] {};
\end{tikzpicture}
\subcaption{$\{132, 231, 312\}$}
\label{fig:132, 231, 312}
\end{subfigure}\begin{subfigure}[b]{0.24\textwidth}
\centering
\begin{tikzpicture}
\draw (0,0) --(2,0) [thick, color=black!20] {};
\draw (0,0) --(0,2) [thick, color=black!20] {};
\draw (0,2) --(2,2) [thick, color=black!20] {};
\draw (2,0) --(2,2) [thick, color=black!20] {};
\draw (1,0) --(1,2) [thick, color=black!20] {};
\draw (0,1) --(2,1) [thick, color=black!20] {};
\draw (0,1) --(1,2) [thick] {};
\draw (1,0) --(2,1) [thick] {};
\end{tikzpicture}
\subcaption{$\{321, 132, 213\}$}
\label{fig:321, 132, 213}
\end{subfigure}\begin{subfigure}[b]{0.24\textwidth}
\centering
\begin{tikzpicture}
\draw (0,0) --(2,0) [thick, color=black!20] {};
\draw (0,0) --(0,2) [thick, color=black!20] {};
\draw (0,2) --(2,2) [thick, color=black!20] {};
\draw (2,0) --(2,2) [thick, color=black!20] {};
\draw (1,0) --(1,2) [thick, color=black!20] {};
\draw (0,1) --(2,1) [thick, color=black!20] {};
\draw (0,1) --(1,0) [thick] {};
\draw (1,2) --(2,1) [thick] {};
\end{tikzpicture}
\subcaption{$\{123, 312, 231\}$}
\label{fig:123, 312, 231}
\end{subfigure}
\caption{ Permutations that avoid the set of patterns listed are exactly those that can be drawn on the corresponding figure. (For more on permutations that can be drawn on figures, see \cite{AABRV13}.) }
\label{fig:3}
\end{figure}
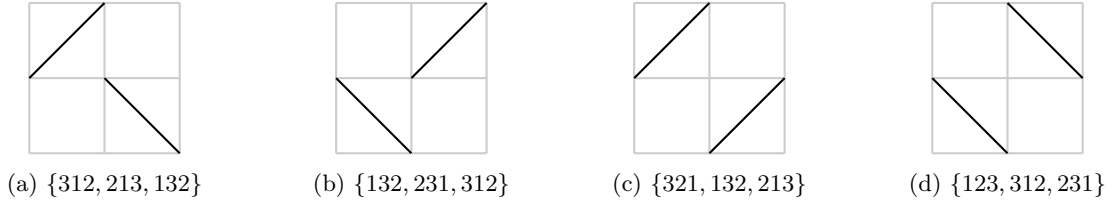

\begin{theorem}\label{thm:uni plus}
For $n\geq 1$, $$f_n(312, 213, 132) = n! \sum_{k=1}^n \frac{1}{k!}{{n-1}\choose{k-1}}.$$
\end{theorem}

\begin{proof}
First, consider a single tree with this property. If the root is $r$, the label $r+1$ must be a child of the root and thus there is one place to add it. Next, $r+2$ can be attached as a child of either $r$ or $r+1$. We continue to build the tree by attaching the remaining labeled vertices in order $r+1, r+2, r+3, \ldots, n, r-1, r-2, \ldots, 2, 1$. There are clearly $(n-1)!$ ways to do this. Since $r$ was chosen arbitrarily, there are exactly $n!$ trees which avoid the set of permutations $\{312, 213, 132\}$.

Further, we can construct an explicit bijection between these trees and permutations. Given a permutation $\pi \in \S_n$, let $\pi_1=r$ be the root of the corresponding tree. For any other element $i$ in the permutation, if $i>r$, let $i$ be a child of the rightmost element $j$ to the left of $i$ so that $i>j\geq r$. If $i<r$, let $i$ be a child of the rightmost element $j$ to the left of $i$ so that $i<j$. The elements greater than $r$ will appear in increasing order, and the elements smaller than $r$ will appear in decreasing order. Further, an element greater than $r$ will never be the child of an element smaller than $r$, and so the requirements are satisfied. To invert this process, write the tree with all children greater than $r$ to the left of all children less than $r$. All children greater than $r$ should appear in increasing order, and to the right of those, all children less than $r$ should appear in decreasing order. Read the labels while traversing the tree clockwise starting at the root. 

Since trees satisfying this property are in bijection with permutations, forests that avoid these permutations are in bijection with partitions of $[n]$ into ordered lists. To obtain such a partition into $k$ parts, we can take a permutation $\pi$ together with a composition $(\lambda_1, \ldots, \lambda_k)$ of $n$ into $k$ parts. For example, if we have permutation $\pi = 35216748$ and composition $(3,1,2,2)$, then we obtain the partition of $[8]$ into lists $352, 1, 67, 48$. The order in which these lists appear does not matter (thus we divide by $k!$), only the order within each list matters. Each list corresponds to a tree in the forest via the bijection described above and thus from a partition into $k$ parts, we obtain a forest on $k$ trees.
\end{proof}

Denote by $\tau$ the bijection described in the proof of Theorem~\ref{thm:uni plus} from partitions of $[n]$ into lists to forests avoiding $\{312, 213, 132\}$. An example can be seen in Figure~\ref{fig:unimodal plus}.

\begin{figure}[h]
\centering
\includegraphics{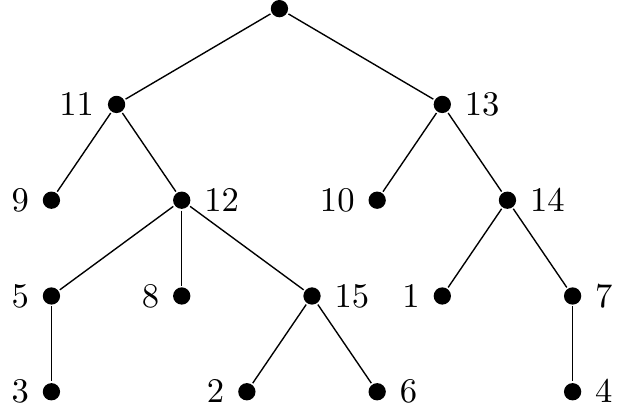}
  \caption{This example is $\tau(\{11, 9, 12, 5, 3, 8, 15, 2, 6\}\{13, 10, 14, 1, 7, 4\})$. }
  \label{fig:unimodal plus}
  \end{figure}
  
  The following corollary follows directly from the proof of Theorem~\ref{thm:uni plus}. 
  \begin{corollary}
  Let $1\leq k\leq n$. The number of forests on $[n]$ that avoid the permutations 312, 213, and 132 and have exactly $k$ trees is equal to $\frac{n!}{k!}{{n-1}\choose{k-1}}$.
  \end{corollary}

By Proposition~\ref{prop:comp} and Theorem~\ref{thm:uni plus}, we obtain the following corollary pertaining to forests that are decreasing, then increasing so that every element of the increasing part is greater than those of the preceding decreasing part. These are exactly the permutations that can be drawn on Figure~\ref{fig:132, 231, 312}.
\begin{corollary}\label{cor:Vplus}
For $n\geq 1$, $$f_n(132, 231, 312) = n! \sum_{k=1}^n \frac{1}{k!}{{n-1}\choose{k-1}}.$$
\end{corollary}


The same argument as above can be applied to the forests that avoid the set of permutations $\{321, 132, 213\}$ (those that can be drawn on Figure~\ref{fig:321, 132, 213}) and those that avoid the set $\{123, 312, 231\}$ (those that can be drawn on Figure~\ref{fig:123, 312, 231}).

\begin{theorem}\label{thm:one descent plus}
For $n\geq 1$, $$f_n(321, 132, 213) =  f_n(123, 312, 231)  = n! \sum_{k=1}^n \frac{1}{k!}{{n-1}\choose{k-1}}.$$
\end{theorem}
\begin{proof}
We begin by considering trees which avoid $\{321, 132, 213\}$. If $r$ is the root, we can attach each of the remaining labels in the order $$r+1, r+2, \ldots, n, 1, 2, \ldots, r-1.$$ Since there are $(n-1)!$ ways to do this and $r$ was chosen arbitrarily, there are $n!$ trees which avoid this set. The result for the set $\{321, 132, 213\}$ follows from the same argument as in the proof of Theorem~\ref{thm:uni plus}. (Note one can also extend this to a bijection, as was done in the proof of Theorem~\ref{thm:uni plus}.) Since $\{123, 312, 231\}$ is the complementary set of permutations, applying Proposition~\ref{prop:comp} completes the proof. 
\end{proof}

\begin{theorem}\label{thm:uni+231}
For $n\geq 1$, let $F(n)$ denote the number of forests on $\left[n\right]$ that avoid the set of permutations $\{213, 312,231\}$.  Then $F(n)$ satisfies the recurrence
\[
F(n)=\sum_{k=1}^n\sum_{r=1}^k{{n-1}\choose{k-1}}(r-1)!F(n-k)F(k-r)
\]
with $F(1)=1$.
\end{theorem}

\begin{proof}
Let $T(n)$ denote the number of trees on $\left[n\right]$ that avoid the set of permutations $\{213, 312,231\}$.  Let $r$ denote the root of such a tree. Since this tree avoids 213 and 231, by removing the root of this tree, we are left with a forest on the set $\{1,2,\ldots, r-1\}$ and a forest on the set $\{r+1,r+2,\ldots,n\}$. Furthermore, since this tree avoids 312, the forest on $\{1,2,\ldots, r-1\}$ must be decreasing. There are $F(n-r)$ forests on $\{r+1,r+2,\ldots,n\}$ that avoids $\{213, 312,231\}$, and there are $(r-1)!$ decreasing forests on $\{1,2,\ldots, r-1\}$.  Since $r$ can be any label from $1$ to $n$, we have
\[
T(n)=\sum_{r=1}^n(r-1)!F(n-r).
\]

Now, let $k$ be the number of vertices in the tree containing $1$.  There are ${{n-1}\choose{k-1}}$ ways to choose the other $k$ vertices in the tree containing $1$, and $T(k)$ ways to construct that tree so that it avoids the set $\{213, 312,231\}$.  Also, there are $F(n-k)$ ways to build a forest on the remaining $n-k$ vertices.  Thus
\[
F(n)=\sum_{k=1}^n{{n-1}\choose{k-1}}T(k)F(n-k).
\]
Substituting for $T(k)$, we obtain the recurrence
\[
F(n)=\sum_{k=1}^n{{n-1}\choose{k-1}}F(n-k)\sum_{r=1}^k(r-1)!F(k-r).
\]
\end{proof}

\begin{corollary}\label{cor:V+213}
For $n\geq 1$, let $F(n)$ denote the number of forests on $\left[n\right]$ that avoid the set of permutations $\{231, 132,213\}$.  Then $F(n)$ satisfies the recurrence
\[
F(n)=\sum_{k=1}^n\sum_{r=1}^k{{n-1}\choose{k-1}}(r-1)!F(n-k)F(k-r)
\]
with $F(1)=1$.
\end{corollary}

\section{Forests avoiding the set $\{321, 2143, 3142\}$}\label{upup}

Next, we consider the labeled rooted forests on $[n]$ for which each path from the root to a vertex has at most one descent. 
These forests are exactly those that avoid the patterns $321$, $2143$, and $3142$.

We start with a lemma. We say that a rooted forest has a \textit{descent} at a vertex if it is greater than at least one of its children, and we say it has a \textit{proper descent} at a vertex if it is greater than all its children. 

\begin{lemma}\label{lem:des}
Let $n\geq 2$. Then the number of trees on $[n]$ that have a proper descent at the root and no other descents is equal to $n!/2$. 
\end{lemma}
\begin{proof}

We can construct a bijection of these trees with permutations of $[n]$ that start with a descent (i.e. in which $\pi_1>\pi_2$) in the following way. Given a permutation with a descent at 1, let $\pi_1$ be the root of the corresponding tree and let all other left-to-right minima be children of the root. For the remaining vertices, let $i$ be the child of the rightmost element $j$ of $\pi$ that precedes $i$ and is less than $i$. Since $\pi$ starts with a descent, we are guaranteed that the only children of the root will be the other left-to-right minima, and thus the tree will have a proper descent at the root. As defined, this clearly will not introduce other descents. To invert this process, write the tree with all children drawn from left to right in increasing order and traverse the tree clockwise starting at the root and reading off labels as they are encountered.  This process is very similar to the bijection $\varphi$ between permutations and increasing trees. Since there are $n!/2$ permutations that start with a descent (indeed, the complement is a bijection between permutations that start with a descent and those that start with an ascent), our proof is complete. 
\end{proof}

We can define a bijection, which we will call $\rho$, from permutations $\pi\in\S_n$ in which 2 appears before 1 to trees on $[n]$ that have a proper descent at the root and no other descents by taking the inverse of $\pi$ and applying the bijection described in the proof of Lemma~\ref{lem:des}. The fact that 2 appears before 1 in the permutation guarantees that the inverse permutation starts with a descent and so we can apply the bijection described above to the inverse permutation. 

In the next proof, we will consider ordered partitions of $[n]$ for which the elements in each subset are also ordered up to reverse.  For example, a few partitions of $[5]$: the partition $\{12\}\{345\}$ is different than both $\{12\}\{435\}$ and $\{345\}\{12\}$ but is the same as $\{21\}\{345\}$ and $\{21\}\{543\}$. For $n\geq 1$, the trees described in Lemma~\ref{lem:des} are in bijection via $\rho$ with these partitions of $[n]$ into one part, since these are clearly permutations which can be written (up to reverse) in which 2 appears before 1. 

\begin{theorem}\label{thm:one descent}
For $n\geq 1$, $$f_n(321, 2143, 3142) = n! \left(1 + \sum_{\substack{1\leq \ell\leq k\leq n \\ \ell+k\leq n}} \frac{1}{2^{\ell}} {{n-k-1}\choose{\ell-1}}{{k}\choose{\ell}}\right).$$
\end{theorem}

\begin{proof}
The proof is similar to the one for unimodal forests. Here, we find a bijection between forests avoiding $\{321, 2143, 3142\}$ and ordered partitions of $[n]$ for which the elements in each subset are also ordered (up to reverse). We can see that these partitions are enumerated by the formula above. Indeed, we can obtain such a partition with $k$ parts where $\ell$ of these parts is of size at least two by taking any permutation of $[n]$ (of which there are $n!$) and a composition of $n$ with $k$ parts of which $\ell$ parts are of size at least two (of which there are ${{n-k-1}\choose{\ell-1}}{{k}\choose{\ell}}$).  The composition determines the sizes of the parts of the partition, and the division by $2^\ell$ accounts for the symmetry under reverse. For example, if $\pi = 351246$ is the permutation and $\lambda = (2,1,3)$ is the composition, then $\{35\}\{1\}\{246\}$ is the corresponding partition. Since $\ell =2$, there are $2^2 = 4$ permutations that result in the same partition (up to reverse of the parts), the other three in this case being $351642$, $531246$, and $531642$. 

The bijection, $\psi$, is as follows. Take such a partition. 
For each of the $k$ parts of the partition, we will obtain a tree. The roots of these trees will constitute the largest increasing subforest of the forest corresponding to the partition. For a part of the partition $\{j\}$ of size one, we obtain the singleton tree with that element. For a part of size at least two, write the ordered list so that the smallest element is to the right of the second smallest element (which can be done by reversing the list if this is not already satisfied) and apply the bijection $\rho$ to each part. For each part, we obtain a tree that is increasing except at the root, which is larger than all its children (but not necessarily all of its descendants). Now, each part of the partition corresponds to some tree and the roots (which are ordered by their position in the ordered partition) of these trees can be arranged into an increasing subforest via $\varphi$. We obtain a forest which has an increasing subforest of size $k$ and which has at most one descent along any path from a root to a vertex. 
\end{proof}

\begin{figure}[h]
\centering
\includegraphics{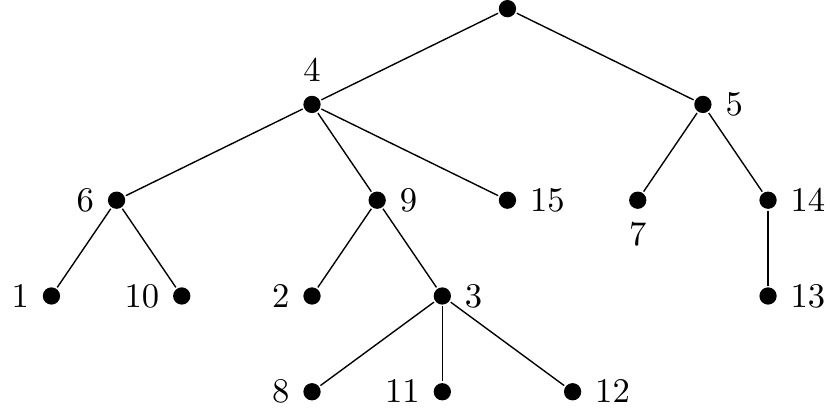}
  \caption{This is an example of a forest on $[15]$ that avoids the three patterns $321$, $2143$, and $3142$. The forest pictured is $\psi(\{5\}\{14,13\} \{7\} \{4\} \{15\} \{12,3,11,2,9,8\} \{6,1\} \{10\})$. Notice that this forest has three descents, but there is only one descent along each path from the root to a given vertex. }
  \label{fig:one descent}
  \end{figure}

The forest in Figure~\ref{fig:one descent} is the image under $\psi$ of the following ordered partition of lists (up to reverse): $$\{5\}\{14,13\} \{7\} \{4\} \{15\} \{12,3,11,2,9,8\} \{6,1\} \{10\}.$$ Here it is written so that in each part, the smallest element is to the right of the second smallest element. For each part, one can obtain a tree on those vertices which has a proper descent at the root vertex and no other descents via the map $\rho$ (which is to take the inverse of the partition listed and create a tree via the map described in the proof of Lemma~\ref{lem:des}). For example, the complement of $12,3, 11, 2, 9, 8$ is $9, 3, 12, 11, 8, 2$. This sequence has a descent in the first position and thus we can apply the map in the proof of Lemma~\ref{lem:des}. The tree obtained is the subtree  composed of vertex 9 and the descendants of 9 in Figure~\ref{fig:one descent}. The roots of the trees obtained constitute the largest increasing subforest of the forest in Figure~\ref{fig:one descent}. In this case, these are the vertices 5, 14, 7, 4, 15, 9, 6, 10. Their arrangement into an increasing forest is obtained via the map $\varphi$.

By Proposition~\ref{prop:comp} and Theorem~\ref{thm:one descent}, we obtain the following corollary pertaining to forests that satisfy the condition that there is at most one ascent per path from a root to a vertex. 
\begin{corollary}\label{cor:downdown}
For $n\geq 1$, $$f_n(123, 3412, 2413) = n! \left(1 + \sum_{\substack{1\leq \ell\leq k\leq n \\ \ell+k\leq n}} \frac{1}{2^{\ell}} {{n-k-1}\choose{\ell-1}}{{k}\choose{\ell}}\right).$$
\end{corollary}

\section{Open questions and future directions}

This new definition of pattern avoidance on forests introduces many open questions that may be interesting to investigate further. One clear direction is to enumerate those forests that avoid a single pattern or other sets of patterns. Many of the questions that can be posed for pattern avoidance of permutations can also be posed in this context. For example, one could study consecutive pattern avoidance or occurrences of patterns in these forests. 

We can define a similar notion of pattern avoidance for ordered forests or forests of binary trees. Some data collected on both classical and consecutive pattern avoidance for these forests can be found below.


\begin{figure}[H]
\centering
\renewcommand{\arraystretch}{1.25}
\begin{tabular}{c|c|c|c|c|c|c}
 &\multicolumn{3}{c|}{Classical} &\multicolumn{3}{c}{Consecutive}\\
 \hline
n & $321$ & $231$ & $132$ &$321$ & $231$ & $132$  \\
\hline
$1$ &$1$ &$1$ &$1$ &$1$ &$1$ &$1$\\
\hline
$2$ &$3$ &$3$ &$3$ &$3$ &$3$ &$3$\\
\hline
$3$ &$14$ &$14$ &$14$ &$14$ &$14$ &$14$\\
\hline
$4$ &$87$ &$87$ &$87$ &$90$ &$89$ &$89$\\
\hline
$5$ &$668$ &$667$ &$668$ &$747$ &$723$ &$722$\\
\end{tabular}
\caption{The number of (unordered) binary forests on $\left[n\right]$ that avoid the given (classical or consecutive) pattern.}
\label{fig:non-planar binary forests}
\end{figure}

\begin{figure}[H]
\centering
\renewcommand{\arraystretch}{1.25}
\begin{tabular}{c|c|c|c|c|c|c}
 &\multicolumn{3}{c|}{Classical} &\multicolumn{3}{c}{Consecutive}\\
 \hline
n & $321$ & $231$ & $132$  & $321$ & $231$ & $132$\\
\hline
$1$ &$1$ &$1$ &$1$ &$1$ &$1$ &$1$\\
\hline
$2$ &$4$ &$4$ &$4$ &$4$ &$4$ &$4$\\
\hline
$3$ &$29$ &$29$ &$29$ &$29$ &$29$ &$29$\\
\hline
$4$ &$304$ &$304$ &$304$ &$307$ &$306$ &$306$\\
\hline
$5$ &$4158$ &$4156$ &$4158$ &$?$ &$?$ &$?$ \\
\end{tabular}
\caption{The number of ordered forests on $\left[n\right]$ that avoid the given (classical or consecutive) pattern.}
\label{fig:plane forets}
\end{figure}


\bibliographystyle{plain}
\bibliography{References_Trees}

\end{document}